\documentclass[12pt,a4paper,oneside]{article}
\usepackage{amsmath,amssymb,amsthm}
\usepackage{amsfonts}
\usepackage{amssymb}
\usepackage{graphicx}
\usepackage{amsthm}
\usepackage{newlfont}
\newlength{\defbaselineskip}
\newcommand{\setlinespacing}[1]%
          {\setlength{\baselineskip}{#1 \defbaselineskip}}

\theoremstyle{plain}
\newtheorem{theorem}{Theorem}[section]

\newtheorem{lemma}[theorem]{Lemma}

\theoremstyle{definition}

\newtheorem{remarks}[theorem]{Remark}
\newtheorem{example}[theorem]{Example}

\begin{document}
\begin{center}
\vspace{3cm}
 {\bf \large Group inverses of $\{0,1\}$-triangular matrices and Fibonacci numbers}\\
\vspace{3cm}

{\bf Manami Chatterjee and K.C. Sivakumar}\\
Department of Mathematics\\
Indian Institute of Technology Madras\\
Chennai 600 036, India. \\
\end{center}

\begin{abstract}
A number $s$ is the sum of the entries of the inverse of an $n \times n, (n \geq 3)$ upper triangular matrix with entries from the set $\{0, 1\}$ if and only if $s$ is an integer lying between $2-F_{n-1}$ and $2+F_{n-1}$, where $F_n$ is the $n$th Fibonacci number. A generalization of the sufficient condition above to singular, group invertible matrices is presented.
\end{abstract}



\newpage
\section{Introduction}
In order to state the main result of this article, we need the notion of the group inverse. Recall that for a matrix $A$ of order $n \times n$ with real entries, the group inverse, if it exists, is the unique matrix $X$ of order $n \times n$ that satisfies the matrix equations $AXA=A, XAX=X$ and $AX=XA$. The group inverse is denoted by $A^{\#}$. A necessary sufficient condition for the group inverse to exist is the condition that $rank(A)=rank(A^2)$. If $A$ is the matrix of order $n \times n$ each of whose entries equals $1$, then $A^{\#}=\frac{1}{n^2}A$, whereas the group inverse does not exist for any nilpotent matrix. For more details, we refer the reader to \cite{bengre}.

Let $S(X)$ denote the sum of the entries of a matrix $X$. Huang, Tam and Wu \cite{Huang} showed that a number $s$ is equal to $S(A^{-1})$ for a symmetric $(0,1)$ matrix $A$ with trace zero if and only if $s$ is rational. Motivated by this work, Farber and Berman \cite{Far} presented an interesting relationship between Fibonacci numbers and matrix theory, thereby providing a partial answer to the question ``what can be said about the sum of the entries of the inverse of a $(0,1)$ matrix?''. They proved: A number $s$ is the sum of the entries of the inverse of an $n \times n, (n \geq 3)$ upper triangular matrix with entries from $\{0, 1\}$ if and only if $s$ is an integer between $2-F_{n-1}$ and $2+F_{n-1}$.  

In this article, we prove that the sufficient condition stated above has a nice extension to singular, group invertible, upper triangular matrices with entries from $\{0, 1\}$. This sufficient condition is also shown to be not a necessary one. 

\section{Preliminary Results}
First, we collect some prelilminary results involving inverses of matrices whose entries come from $\{0,1\}$. The inverses have as their entries, Fibonacci numbers. It is well known that, frequently, computations involving Fibonacci numbers
are long and tedious. Our experience is similar and so we have postponed the proofs of the basic results to the Appendix. The following basic formulae will be used quite often: $$1+{\sum_{i=1}^{n}{F_{2i}}}=F_{2n+1} \ \  and \ \ \sum_{i=1}^{n}{F_{2i-1}}=F_{2n}.$$ 

In the first four results, we present formulae for four matrices and their inverses. These will play a crucial role in our discussion. 

\begin{lemma}\label{c1}
Let a square matrix $C_1 $ of order $n-1$ ($n\geq 6$, $n$ even) be defined as:
\begin{eqnarray}
 C_1 e^1=e^1,\quad C_1 e^2=e^2.\label{c1e1} 
 \end{eqnarray}
 For odd $k, ~3 \leq k \leq n-3$, let 
 \begin{eqnarray}
 C_1 e^k&=&e^1+\sum_{i=1}^{\frac{k-1}{2}}e^{2i}+e^k \label{c1odd}
 \end{eqnarray}
 and for $k$ even, $4\leq k \leq n-2$, let 
  \begin{eqnarray}
 C_1 e^k&=&\sum_{i=1}^{\frac{k-2}{2}}e^{2i+1}+e^k. \label{c1even}   
\end{eqnarray}
Finally, let 
\begin{eqnarray}
C_1 e^{n-1}&=&\sum_{i=1}^{\frac{n-4}{2}}e^{2i+1}+e^{n-1}.\label{c1(n-1)}
\end{eqnarray}
Let $X$ be square of order ${n-1}$ such that
\begin{equation}
Xe^1=e^1,\quad Xe^2=e^2.\label{X1e1}
\end{equation}
For odd $k, ~3\leq k\leq n-3$, let
\begin{equation}
Xe^k = -F_{k-2} e^1 + \sum_{i=2}^{k-1}{(-1)^{i+1}F_{k-i}e^i + e^k}, \label{X1odd}
\end{equation}
for even $k, ~4\leq k\leq n-2$, let
\begin{equation}
Xe^k=F_{k-2} e^1 + \sum_{i=2}^{k-1}{(-1)^{i}F_{k-i}e^i + e^k}\label{X1even}
\end{equation}
and let
\begin{equation}
Xe^{n-1}=F_{n-4} e^1 + \sum_{i=2}^{n-3}{(-1)^{i}F_{n-2-i}e^i + e^{n-1}}.\label{X1(n-1)}
\end{equation}
Then, $X = C_1^{-1}$.
\end{lemma}

\begin{example}\label{c1ex}
We illustrate the lemma above, by two examples. For $n=6$,\\
$$C_1=\left( \begin{array}{ccccc}
1 & 0 & 1 & 0 & 0\\
0 & 1 & 1 & 0 & 0\\
0 & 0 & 1 & 1 & 1\\
0 & 0 & 0 & 1 & 0\\
0 & 0 & 0 & 0 & 1
\end{array} \right) \ \ and \ \  C_1^{-1}=\left( \begin{array}{ccccc}
1 & 0 & -F_1 & F_2 & F_2\\
0 & 1 & -F_1 & F_2 & F_2\\
0 & 0 & 1 & -F_1 & -F_1\\
0 & 0 & 0 & 1 & 0\\
0 & 0 & 0 & 0 & 1
\end{array} \right).$$

For $n=8$,\\
$$C_1=\left( \begin{array}{ccccccc}
1 & 0 & 1 & 0 & 1 & 0 & 0\\
0 & 1 & 1 & 0 & 1 & 0 & 0\\
0 & 0 & 1 & 1 & 0 & 1 & 1\\
0 & 0 & 0 & 1 & 1 & 0 & 0\\
0 & 0 & 0 & 0 & 1 & 1 & 1\\
0 & 0 & 0 & 0 & 0 & 1 & 0\\
0 & 0 & 0 & 0 & 0 & 0 & 1
\end{array} \right),$$

while

$$C_1^{-1}=\left( \begin{array}{ccccccc}
1 & 0 & -F_1 & F_2 & -F_3 & F_4 & F_4\\
0 & 1 & -F_1 & F_2 & -F_3 & F_4 & F_4\\
0 & 0 & 1 & -F_1 & F_2 & -F_3 & -F_3\\
0 & 0 & 0 & 1 & -F_1 & F_2 & F_2\\
0 & 0 & 0 & 0 & 1 & -F_1 & -F_1\\
0 & 0 & 0 & 0 & 0 & 1 & 0\\
0 & 0 & 0 & 0 & 0 & 0 & 1
\end{array} \right).$$
\end{example}

\begin{lemma}\label{c2}
Let $ C_2 $ be a square matrix of order $n-1$ ($n\geq 6$, $n$ is odd) defined as:
\begin{eqnarray}
 C_2 e^1=e^1, \quad  C_2 e^2=e^2.\label{c2e1}
 \end{eqnarray}
 For odd $k,~ 3\leq k\leq n-4$, let
 \begin{equation}
 C_2 e^k=e^1+\sum_{i=1}^{\frac{k-1}{2}}e^{2i}+e^k \label{c2odd}
 \end{equation}
 and for even $k,~4\leq k \leq n-3$ let
 \begin{equation}
 C_2 e^k=\sum_{i=1}^{\frac{k-2}{2}}e^{2i+1}+e^k.\label{c2even} 
 \end{equation}
 Also let
 \begin{eqnarray}
 C_2 e^{n-2}&=&\sum_{i=1}^{\frac{n-5}{2}}e^{2i+1}+e^{n-2}\label{c2(n-2)}
 \end{eqnarray}
 and
 \begin{eqnarray}
  C_2 e^{n-1}&=&\sum_{i=1}^{\frac{n-5}{2}}e^{2i+1}+e^{n-1}.\label{c2(n-1)}
\end{eqnarray}
 Let $X$ be of order ${n-1}$ and be defined such that 
\begin{eqnarray}
 Xe^1=e^1,\quad Xe^2=e^2.\label{X2e1}
 \end{eqnarray}
 For odd $k,~ 3\leq k\leq n-4$, let
 \begin{eqnarray}
 Xe^k&=&-F_{k-2} e^1 + \sum_{i=2}^{k-1}{(-1)^{i+1}F_{k-i}e^i} + e^k\label{X2odd}
 \end{eqnarray}
 and for even $k,~4\leq k \leq n-3$ let
 \begin{eqnarray}
 Xe^k&=&F_{k-2} e^1 + \sum_{i=2}^{k-1}{(-1)^{i}F_{k-i}e^i + e^k}.\label{X2even}
 \end{eqnarray}
 Further let
\begin{eqnarray}
 Xe^{n-2}&=&F_{n-5} e^1 + \sum_{i=2}^{n-4}{(-1)^{i}F_{n-3-i}e^i + e^{n-2}}\label{X2(n-2)}
 \end{eqnarray}
 and 
 \begin{eqnarray}
 Xe^{n-1}&=&F_{n-5} e^1 + \sum_{i=2}^{n-4}{(-1)^{i}F_{n-3-i}e^i + e^{n-1}}.\label{X2(n-1)}
 \end{eqnarray}
 Then $X = C_2^{-1}$.
\end{lemma}

\begin{example}\label{c2ex}
Let us give an example in support of Lemma (\ref{c2}). For $n=7$,\\
$$C_2=\left( \begin{array}{cccccc}
1 & 0 & 1 & 0 & 0 & 0\\
0 & 1 & 1 & 0 & 0 & 0\\
0 & 0 & 1 & 1 & 1 & 1\\
0 & 0 & 0 & 1 & 0 & 0\\
0 & 0 & 0 & 0 & 1 & 0\\
0 & 0 & 0 & 0 & 0 & 1
\end{array} \right),$$ while

$$C_2^{-1}=\left( \begin{array}{cccccc}
1 & 0 & -F_1 & F_2 & F_2 & F_2\\
0 & 1 & -F_1 & F_2 & F_2 & F_2\\
0 & 0 & 1 & -F_1 & -F_1 & -F_1\\
0 & 0 & 0 & 1 & 0 & 0\\
0 & 0 & 0 & 0 & 1 & 0\\
0 & 0 & 0 & 0 & 0 & 1
\end{array} \right).$$
\end{example}

\begin{lemma}\label{c3}
Let $C_3$ be a square matrix of order $n-1$ ($n\geq 6$) with $n$ being even, defined as:
\begin{eqnarray}
 C_3 e^1=e^1, \quad C_3 e^2=e^2.\label{c3e1}
 \end{eqnarray}
 For odd $k,~  3\leq k \leq n-3$, let
 \begin{eqnarray}
 C_3 e^k&=&e^1+\sum_{i=1}^{\frac{k-1}{2}}e^{2i}+e^k \label{c3odd}
 \end{eqnarray}
 and for even $k,~ 4\leq k \leq n-4$, let
 \begin{eqnarray}
 C_3 e^k&=&\sum_{i=1}^{\frac{k-2}{2}}e^{2i+1}+e^k.\label{c3even}
 \end{eqnarray}
 Also let
 \begin{eqnarray}
 C_3 e^{n-2}&=&e^1+\sum_{i=1}^{\frac{n-4}{2}}e^{2i}+e^{n-2}\label{c3(n-2)}
 \end{eqnarray}
 and 
 \begin{eqnarray}
  C_3 e^{n-1}&=&e^1+\sum_{i=1}^{\frac{n-4}{2}}e^{2i}+e^{n-1}.\label{c3(n-1)}
 \end{eqnarray}
 Let $Y$ be defined as 
\begin{eqnarray}
 Ye^1=e^1,\quad Ye^2=e^2.\label{Y1e1}
 \end{eqnarray}
 For odd $k,~3\leq k \leq n-3$
 \begin{eqnarray}
 Ye^k&=&-F_{k-2} e^1 + \sum_{i=2}^{k-1}{(-1)^{i+1}F_{k-i}e^i} + e^k\label{Y1odd}
 \end{eqnarray}
 and for even $k,~4\leq k \leq n-4$
 \begin{eqnarray}
 Ye^k&=&F_{k-2} e^1 + \sum_{i=2}^{k-1}{(-1)^{i}F_{k-i}e^i + e^k}.\label{Y1even}
 \end{eqnarray}
 Further let
 \begin{eqnarray}
 Ye^{n-2}&=&-F_{n-5} e^1 + \sum_{i=2}^{n-4}{(-1)^{i+1}F_{n-3-i}e^i + e^{n-2}}\label{Y1(n-2)}
 \end{eqnarray}
 and 
 \begin{eqnarray}
 Ye^{n-1}&=&-F_{n-5} e^1 + \sum_{i=2}^{n-4}{(-1)^{i+1}F_{n-3-i}e^i + e^{n-1}}.\label{Y1(n-1)}
 \end{eqnarray}
 Then $Y = C_3^{-1}$.
 \end{lemma}
 
\begin{example}\label{c3ex}
Here is an illustration for $n=8$.\\
$$C_3=\left( \begin{array}{ccccccc}
1 & 0 & 1 & 0 & 1 & 1 & 1\\
0 & 1 & 1 & 0 & 1 & 1 & 1\\
0 & 0 & 1 & 1 & 0 & 0 & 0\\
0 & 0 & 0 & 1 & 1 & 1 & 1\\
0 & 0 & 0 & 0 & 1 & 0 & 0\\
0 & 0 & 0 & 0 & 0 & 1 & 0\\
0 & 0 & 0 & 0 & 0 & 0 & 1
\end{array} \right)$$ and 
$$C_3^{-1}=\left( \begin{array}{ccccccc}
1 & 0 & -F_1 & F_2 & -F_3 & -F_3 & -F_3\\
0 & 1 & -F_1 & F_2 & -F_3 & -F_3 & -F_3\\
0 & 0 & 1 & -F_1 & F_2 & F_2 & F_2\\
0 & 0 & 0 & 1 & -F_1 & -F_1 & -F_1\\
0 & 0 & 0 & 0 & 1 & 0 & 0\\
0 & 0 & 0 & 0 & 0 & 1 & 0\\
0 & 0 & 0 & 0 & 0 & 0 & 1
\end{array} \right).$$
\end{example}

\begin{lemma}\label{c4}
Let $C_4$ be a square matrix of order $n-1$ ($n\geq 6$, $n$ odd) be defined as:
\begin{eqnarray}
 C_4 e^1=e^1, \quad C_4 e^2=e^2.
 \end{eqnarray}
 For odd $k,~3\leq k \leq n-2$, let
 \begin{eqnarray}
 C_4 e^k&=&e^1+\sum_{i=1}^{\frac{k-1}{2}}e^{2i}+e^k,
 \end{eqnarray}
 for even $k,~4\leq k \leq n-3$, let
 \begin{eqnarray}
 C_4 e^k&=& \sum_{i=1}^{\frac{k-2}{2}}e^{2i+1}+e^k
 \end{eqnarray}
 and
 \begin{eqnarray}
 C_4 e^{n-1}&=&e^1+\sum_{i=1}^{\frac{n-3}{2}}e^{2i}+e^{n-1}.
 \end{eqnarray}
Let $Y$ of order ${n-1}$ be defined such that 
  \begin{eqnarray}
 Ye^1=e^1,\quad Ye^2=e^2.
 \end{eqnarray}
 For odd $k,~3\leq k \leq n-2$, let
 \begin{eqnarray}
 Ye^k&=&-F_{k-2} e^1 + \sum_{i=2}^{k-1}{(-1)^{i+1}F_{k-i}e^i}+ e^k,
 \end{eqnarray}
 for even $k,~4\leq k \leq n-3$, let
 \begin{eqnarray}
 Ye^k&=&F_{k-2} e^1 + \sum_{i=2}^{k-1}(-1)^{i}F_{k-i}e^i+e^k
 \end{eqnarray}
 and
 \begin{eqnarray}
Ye^{n-1}&=&-F_{n-4} e^1 + \sum_{i=2}^{n-3}{(-1)^{i+1}F_{n-2-i}e^i + e^{n-1}}.
\end{eqnarray}
Then, $Y = C_4^{-1}$.
 \end{lemma}
 
\begin{example}\label{c4ex}
Let $n=7$.\\
$$C_4=\left( \begin{array}{cccccc}
1 & 0 & 1 & 0 & 1 & 1\\
0 & 1 & 1 & 0 & 1 & 1\\
0 & 0 & 1 & 1 & 0 & 0\\
0 & 0 & 0 & 1 & 1 & 1\\
0 & 0 & 0 & 0 & 1 & 0\\
0 & 0 & 0 & 0 & 0 & 1
\end{array} \right),$$ whereas, 
$$C_4^{-1}=\left( \begin{array}{cccccc}
1 & 0 & -F_1 & F_2 & -F_3 & -F_3\\
0 & 1 & -F_1 & F_2 & -F_3 & -F_3\\
0 & 0 & 1 & -F_1 & F_2 & F_2\\
0 & 0 & 0 & 1 & -F_1 & -F_1\\
0 & 0 & 0 & 0 & 1 & 0\\
0 & 0 & 0 & 0 & 0 & 1
\end{array} \right).$$
\end{example}

We need the column sums of the four matrices $C_1^{-1}, C_2^{-1}, C_3^{-1}$ and $C_4^{-1}$. These are collected in the next two results. In the first result to follow, we give the required numbers for the first two matrices. 

\begin{lemma}\label{colsum1}
For the matrix $C_1$, we have: 
\begin{eqnarray}
e^TC_1^{-1}e^1=1,\quad e^TC_1^{-1}e^2=1.\label{colsum1e1}
\end{eqnarray}
For odd $k,~3\leq k \leq n-3$,
\begin{eqnarray}
e^TC_1^{-1}e^k&=&-F_{k-1},\label{colsum1odd}
\end{eqnarray}
for even $k,~4\leq k \leq n-2$,
\begin{eqnarray}
e^TC_1^{-1}e^k&=&F_{k-1},\label{colsum1even}
\end{eqnarray}
and
\begin{eqnarray}
 e^TC_1^{-1}e^{n-1}&=&F_{n-3}.\label{colsum1(n-1)}
\end{eqnarray}
For $C_2$, we have: 
\begin{eqnarray}
e^TC_2^{-1}e^1=1,\quad e^TC_2^{-1}e^2=1.
\end{eqnarray}
For odd $k,~3\leq k \leq n-4$,
\begin{eqnarray}
e^TC_2^{-1}e^k&=&-F_{k-1}
\end{eqnarray}
and for even $k,~4\leq k \leq n-3$
\begin{eqnarray}
e^TC_2^{-1}e^k&=&F_{k-1}.
\end{eqnarray}
Further
\begin{eqnarray}
e^TC_2^{-1}e^{n-2}&=&F_{n-4} 
\end{eqnarray}
and 
\begin{eqnarray}
 e^TC_2^{-1}e^{n-1}=F_{n-4}.
\end{eqnarray}
\end{lemma}

Next, we turn our attention to the next two matrices.

\begin{lemma}
For the matrix $C_3$, we have:
\begin{eqnarray}
e^TC_3^{-1}e^1=1,\quad e^TC_3^{-1}e^2=1.
\end{eqnarray}
For odd $k,~3 \leq k \leq {n-3}$,
\begin{eqnarray}
e^TC_3^{-1}e^k&=&-F_{k-1}
\end{eqnarray}
and for even $k,~4 \leq k \leq {n-4}$,
\begin{eqnarray}
e^TC_3^{-1}e^k&=&F_{k-1}.
\end{eqnarray}
Further
\begin{eqnarray}
e^TC_3^{-1}e^{n-2}&=&-F_{n-4}
\end{eqnarray}
and 
\begin{eqnarray}
 e^TC_3^{-1}e^{n-1}=-F_{n-4}.
\end{eqnarray}
For $C_4$, we have: 
\begin{eqnarray}
e^TC_4^{-1}e^1=1,\quad e^TC_4^{-1}e^2=1.
\end{eqnarray}
For odd $k,~3 \leq k \leq {n-2}$,
\begin{eqnarray}
e^TC_4^{-1}e^k&=&-F_{k-1},
\end{eqnarray}
for even $k,~4 \leq k \leq {n-3}$
\begin{eqnarray}
e^TC_4^{-1}e^k&=&F_{k-1}
\end{eqnarray}
and 
\begin{eqnarray}
e^TC_4^{-1}e^{n-1}&=&-F_{n-3}.
\end{eqnarray}
\end{lemma}

\begin{remarks}
In Example \ref{c1ex}, for the second matrix ($n=8$) one may observe that $e^TC_1^{-1}e^1=1=e^TC_1^{-1}e^2, e^TC_1^{-1}e^3=-1=-F_2, e^TC_1^{-1}e^4=2=F_3, e^TC_1^{-1}e^5=-3=-F_4, e^TC_1^{-1}e^6=5=F_5$ and $e^TC_1^{-1}e^7=5=F_5$. 
\end{remarks}

In the next result, we provide formulae for the sum of all the entries of the inverses of the four matrices considered above.

\begin{lemma}\label{sumofall}
$$S(C_1^{-1})=2+F_{n-2}= S(C_2^{-1})$$
and 
$$S(C_3^{-1})=2-F_{n-2}=S(C_4^{-1}).$$
\end{lemma}

In what follows, we calculate the column sums of the matrices $C_1^{-2}, C_2^{-2}, C_3^{-2}$ and $C_4^{-2}$.

The $k$th column sum of $C_1^{-2}$
\begin{eqnarray}
&=&e^TC_1^{-2}e^k \nonumber \\
&=&e^TC_1^{-1}I_{n-1}C_1^{-1}e^k \nonumber \\
&=&e^TC_1^{-1}(\sum_{i=1}^{n-1}e^i(e^i)^T)C_1^{-1}e^k\nonumber \\
&=&\sum_{i=1}^{n-1}(e^TC_1^{-1}e^i)((e^i)^TC_1^{-1}e^k).
\end{eqnarray}
Now, $e^TC_1^{-1}e^i$ is the $i$ th column sum of $C_1^{-1}$ and $(e^i)^TC_1^{-1}e^k$ is nothing but the $(i,k)$-th entry of $C_1^{-1}$. We already have the formula of $C_1^{-1}$ and we have calculated the column sums of $C_1^{-1}$ also. The column sums of $C_1^{-2}$ are determined as follows.\\
Since it is clear that $C_1^{-1}$ is an upper triangular matrix, it follows that,
\begin{equation}
(e^i)^TC_1^{-1}e^k=0\quad \textrm{for}\quad i>k.\nonumber
\end{equation}
 Further,
\begin{eqnarray}
(e^1)^TC_1^{-1}e^2=0\quad \textrm{and}\quad(e^{n-2})^TC_1^{-1}e^{n-1}=0,
\end{eqnarray}
by (\ref{X1e1}) and (\ref{X1(n-1)}). So,
\begin{eqnarray}
e^TC_1^{-2}e^1&=&\sum_{i=1}^{n-1}(e^TC_1^{-1}e^i)((e^i)^TC_1^{-1}e^1)\nonumber\\
&=&(e^TC_1^{-1}e^1)((e^1)^TC_1^{-1}e^1)      \nonumber\\
&=&1\label{alpha1}, 
\end{eqnarray}
by (\ref{X1e1}) and (\ref{colsum1e1}).
\begin{eqnarray}
e^TC_1^{-2}e^2&=&\sum_{i=1}^{n-1}(e^TC_1^{-1}e^i)((e^i)^TC_1^{-1}e^2)\nonumber
\end{eqnarray}
\begin{eqnarray}
&=&(e^TC_1^{-1}e^2)((e^2)^TC_1^{-1}e^2)       \nonumber\\
&=&1,\label{alpha2}
\end{eqnarray}
by (\ref{X1e1}) and (\ref{colsum1e1}).\\
For odd $k$ ($3\leq k \leq n-3$),
\begin{eqnarray}
e^TC_1^{-2}e^k&=&\sum_{i=1}^{n-1}(e^TC_1^{-1}e^i)((e^i)^TC_1^{-1}e^k)\nonumber\\
&=&(e^TC_1^{-1}e^1)((e^1)^TC_1^{-1}e^k)+\sum_{i=1}^{\frac{k-1}{2}}(e^TC_1^{-1}e^{2i})((e^{2i})^TC_1^{-1}e^k)\nonumber\\
&+&\sum_{i=1}^{\frac{k-3}{2}}(e^TC_1^{-1}e^{2i+1})((e^{2i+1})^TC_1^{-1}e^k)+(e^TC_1^{-1}e^k)((e^k)^TC_1^{-1}e^k)\nonumber\\
&=&-F_{k-2}-\sum_{i=1}^{\frac{k-1}{2}}F_{2i-1}F_{k-2i}-\sum_{i=1}^{\frac{k-3}{2}}F_{2i}F_{k-(2i+1)}-F_{k-1}\label{alphaodd}, 
\end{eqnarray}
by (\ref{X1odd}), (\ref{colsum1e1}), (\ref{colsum1odd}) and (\ref{colsum1even}).\\
For even $k$ ($4\leq k \leq n-2$),
\begin{eqnarray}
e^TC_1^{-2}e^k&=&\sum_{i=1}^{n-1}(e^TC_1^{-1}e^i)((e^i)^TC_1^{-1}e^k)\nonumber\\
&=&(e^TC_1^{-1}e^1)((e^1)^TC_1^{-1}e^k)+\sum_{i=1}^{\frac{k-2}{2}}(e^TC_1^{-1}e^{2i})((e^{2i})^TC_1^{-1}e^k)\nonumber\\
&+&\sum_{i=1}^{\frac{k-2}{2}}(e^TC_1^{-1}e^{2i+1})((e^{2i+1})^TC_1^{-1}e^k)+(e^TC_1^{-1}e^k)((e^k)^TC_1^{-1}e^k)\nonumber\\
&=&F_{k-2}+\sum_{i=1}^{\frac{k-2}{2}}F_{2i-1}F_{k-2i}+\sum_{i=1}^{\frac{k-2}{2}}F_{2i}F_{k-(2i+1)}+F_{k-1}\label{alphaeven}, 
\end{eqnarray}
by (\ref{X1even}), (\ref{colsum1e1}), (\ref{colsum1odd}) and (\ref{colsum1even}).
\begin{eqnarray}
e^TC_1^{-2}e^{n-1}&=&\sum_{i=1}^{n-1}(e^TC_1^{-1}e^i)((e^i)^TC_1^{-1}e^{n-1})\nonumber
\end{eqnarray}
\begin{eqnarray}
&=&(e^TC_1^{-1}e^1)((e^1)^TC_1^{-1}e^{n-1})+\sum_{i=1}^{\frac{n-2}{2}}(e^TC_1^{-1}e^{2i})((e^{2i})^TC_1^{-1}e^{n-1})\nonumber\\
&+&\sum_{i=1}^{\frac{n-2}{2}}(e^TC_1^{-1}e^{2i+1})((e^{2i+1})^TC_1^{-1}e^{n-1})\nonumber\\
&=&F_{n-4}+\sum_{i=1}^{\frac{n-4}{2}}(F_{2i-1})(F_{n-2-2i})+\sum_{i=1}^{\frac{n-4}{2}}(-F_{2i})(-F_{n-2-(2i+1)})\nonumber\\
&+&(e^TC_1^{-1}e^{n-2})((e^{n-2})^TC_1^{-1}e^{n-1}) +(e^TC_1^{-1}e^{n-1})((e^{n-1})^TC_1^{-1}e^{n-1})\nonumber\\
&=&F_{n-4}+\sum_{i=1}^{\frac{n-4}{2}}F_{2i-1}F_{n-2-2i}+\sum_{i=1}^{\frac{n-4}{2}}F_{2i}F_{n-2-(2i+1)}+F_{n-3}\label{alpha(n-1)},
\end{eqnarray}
using (\ref{X1(n-1)}), (\ref{colsum1e1}), (\ref{colsum1odd}), (\ref{colsum1even}) and (\ref{colsum1(n-1)}).\\
Similarly, the $k$th column sum of $C_2^{-2}=\sum_{i=1}^{n-1}(e^TC_2^{-1}e^i)((e^i)^TC_2^{-1}e^k)$. Also, $C_2^{-1}$ is an upper triangular matrix. Further,
\begin{eqnarray}
(e^1)^TC_2^{-1}e^2&=&0,\\
(e^{n-3})^TC_2^{-1}e^{n-2}&=&0,\\
(e^{n-3})^TC_2^{-1}e^{n-1}&=&0,\\
\textrm{and}\quad(e^{n-2})^TC_2^{-1}e^{n-1}&=&0,
\end{eqnarray}
using the formulae of $C_2^{-1}$ from Lemma (\ref{c2}). Now, by similar computations, using Lemma 2.2, 2.5 and the four equations as above, one can show the following: 
\begin{eqnarray}
e^TC_2^{-2}e^1=1\quad \textrm{and}\quad e^TC_2^{-2}e^2=1.
\end{eqnarray}
For odd $k$ ($3\leq k \leq n-4$),
\begin{eqnarray}
e^TC_2^{-2}e^k&=&-F_{k-2}-\sum_{i=1}^{\frac{k-1}{2}}F_{2i-1}F_{k-2i}-\sum_{i=1}^{\frac{k-3}{2}}F_{2i}F_{k-(2i+1)}-F_{k-1}.
\end{eqnarray}
For even $k$ ($4\leq k \leq n-3$),
\begin{eqnarray}
e^TC_2^{-2}e^k&=&F_{k-2}+\sum_{i=1}^{\frac{k-2}{2}}F_{2i-1}F_{k-2i}+\sum_{i=1}^{\frac{k-2}{2}}F_{2i}F_{k-(2i+1)}+F_{k-1}.\\
e^TC_2^{-2}e^{n-2}&=&F_{n-5}+\sum_{i=1}^{\frac{n-5}{2}}F_{2i-1}F_{n-3-2i}+\sum_{i=1}^{\frac{n-5}{2}}F_{2i}F_{n-3-(2i+1)}+F_{n-4}.\\
e^TC_2^{-2}e^{n-1}&=&F_{n-5}+\sum_{i=1}^{\frac{n-5}{2}}F_{2i-1}F_{n-3-2i}+\sum_{i=1}^{\frac{n-5}{2}}F_{2i}F_{n-3-(2i+1)}+F_{n-4}.
\end{eqnarray}
A similar argument applies to $C_3$ and $C_4$, too. The formulae corresponding to the column sums of $C_3^{-2}$ and $C_4^{-2}$ are as follows:
\begin{eqnarray}
e^TC_3^{-2}e^1=1\quad \textrm{and}\quad e^TC_3^{-2}e^2=1.
\end{eqnarray}
When $k$ is odd, ($3\leq k \leq {n-3}$)
\begin{eqnarray}
e^TC_3^{-2}e^k&=& -F_{k-2}-\sum_{i=1}^{\frac{k-1}{2}}F_{2i-1}F_{k-2i}-\sum_{i=1}^{\frac{k-3}{2}}F_{2i}F_{k-(2i+1)}-F_{k-1}.
\end{eqnarray}
When $k$ is even, ($4\leq k \leq {n-4}$)
\begin{eqnarray}
e^TC_3^{-2}e^k&=& F_{k-2}+\sum_{i=1}^{\frac{k-2}{2}}F_{2i-1}F_{k-2i}+\sum_{i=1}^{\frac{k-2}{2}}F_{2i}F_{k-(2i+1)}+F_{k-1}.
\end{eqnarray}
\begin{eqnarray}
e^TC_3^{-2}e^{n-2}&=& -F_{n-5}-\sum_{i=1}^{\frac{n-4}{2}}F_{2i-1}F_{n-3-2i}-\sum_{i=1}^{\frac{n-6}{2}}F_{2i}F_{n-3-(2i+1)}\nonumber\\&-&F_{n-4}.\\
e^TC_3^{-2}e^{n-1}&=& -F_{n-5}-\sum_{i=1}^{\frac{n-4}{2}}F_{2i-1}F_{n-3-2i}-\sum_{i=1}^{\frac{n-6}{2}}F_{2i}F_{n-3-(2i+1)}\nonumber\\&-&F_{n-4}.
\end{eqnarray}
\begin{eqnarray}
e^TC_4^{-2}e^1=1\quad\textrm{and}\quad e^TC_4^{-2}e^2=1.
\end{eqnarray}
When $k$ is odd, ($3\leq k \leq {n-2}$)
\begin{eqnarray}
e^TC_4^{-2}e^k&=& -F_{k-2}-\sum_{i=1}^{\frac{k-1}{2}}F_{2i-1}F_{k-2i}-\sum_{i=1}^{\frac{k-3}{2}}F_{2i}F_{k-(2i+1)}-F_{k-1}.
\end{eqnarray}
When $k$ is even, ($4\leq k \leq {n-3}$)
\begin{eqnarray}
e^TC_4^{-2}e^k&=& F_{k-2}+\sum_{i=1}^{\frac{k-2}{2}}F_{2i-1}F_{k-2i}+\sum_{i=1}^{\frac{k-2}{2}}F_{2i}F_{k-(2i+1)}+F_{k-1}.\\
e^TC_4^{-2}e^{n-1}&=&-F_{n-4}-\sum_{i=1}^{\frac{n-3}{2}}F_{2i-1}F_{n-2-2i}-\sum_{i=1}^{\frac{n-5}{2}}F_{2i}F_{n-2-(2i+1)}\nonumber\\&-&F_{n-3}.
\end{eqnarray}

Observe that the absolute $k$th column sum of $C_1^{-2}$ and $C_2^{-2}$ are the same except the last two columns. Let us denote the absolute $k$th column sum of $C_1^{-2}$ to be $\alpha_k$. We can easily see that 
\begin{equation}
e^TC_1^{-2}e^{n-2}=e^TC_1^{-2}e^{n-1}\nonumber
\end{equation}
and
\begin{equation}
e^TC_2^{-2}e^{n-3}=e^TC_2^{-2}e^{n-2}
=e^TC_2^{-2}e^{n-1}.\nonumber
\end{equation}

Similarly, the absolute $k$th column sum of $C_3^{-2}$ and $C_4^{-2}$ are the same, except for the last two columns. Let us denote the absolute $k$th column sum of $C_3^{-2}$ to be $\beta_k$. Further, we have
\begin{equation}
e^TC_3^{-2}e^{n-3}=e^TC_3^{-2}e^{n-2}
=e^TC_3^{-2}e^{n-1}\nonumber
\end{equation}
and
\begin{equation}
e^TC_4^{-2}e^{n-2}=e^TC_4^{-2}e^{n-1}\nonumber.
\end{equation}

The following two important properties of the two sequences $\alpha_k$ and $\beta_k$, will prove to be useful.

\begin{lemma}\label{twoseqres}
The sequences $\alpha_k$ and $\beta_k$ satisfy the following inequalities:\\
$(i) ~\alpha_k \leq \alpha_{k+1} \leq 2\alpha_k$.\\
$(ii)~\beta_k \leq \beta_{k+1} \leq 2\beta_k$.
\end{lemma}

\section{Main Result}
Here, we prove the main result of this article.

\begin{theorem}\label{mainresult}
	For every $n\geq 6$, let $C_1$, $C_2$, $C_3$ and $C_4$ be the matrices described earlier.
	Let $p_n=S(C_1^{-2}u^n)$, where
	\begin{eqnarray}
	u^n&=&e^1+\sum_{i=1}^{\frac{n-2}{2}}e^{2i}+e^{n-1}, \textrm {$n$ is even}. \nonumber 
	\end{eqnarray}
	Let $r_n=S(C_2^{-2}v^n)$, where
	\begin{eqnarray}
	v^n&=&e^1+\sum_{i=1}^{\frac{n-3}{2}}e^{2i}+e^{n-2}+e^{n-1}, \textrm {$n$ is odd}. \nonumber
	\end{eqnarray}
	Let $q_n=-S(C_3^{-2}w^n)$, where
	\begin{eqnarray}
	w^n&=&\sum_{i=1}^{\frac{n-4}{2}}e^{2i+1}+e^{n-2}+e^{n-1}, \textrm { $n$ is even.}\nonumber 
	\end{eqnarray}
	Let $s_n=-S(C_4^{-2}z^n)$, where
	\begin{eqnarray}
	z^n&=&\sum_{i=1}^{\frac{n-3}{2}}e^{2i+1}+e^{n-1}, \quad \quad \textrm {where $n$ is odd.} \nonumber
	\end{eqnarray}
	Let $s$ be an integer satisfying either:
	\begin{eqnarray}
	2-F_{n-2}-q_n \leq s \leq 2+F_{n-2}+p_n,  \nonumber
	\end{eqnarray}
	or
	\begin{eqnarray}
	2-F_{n-2}-s_n \leq s \leq 2+F_{n-2}+r_n.  \nonumber
	\end{eqnarray}
	Then there exists an upper triangular, $\{0, 1\}$, singular, group invertible matrix $A$ of order $n$ such that $S(A^\#)=s$.
\end{theorem}
\begin{proof}
Observe that all the even column sums, as well as the first and the last column sums of $C_1^{-2}$ are positive and so $p_n$ is positive. All the even column sums, the first and the last two column sums of $C_2^{-2}$ also are positive. So, $r_n$ is positive. Again for $C_3^{-2}$ all the odd column sums and the last two column sums are negative and so $q_n$ is positive. For a similar reason $s_n$ is positive. 

First, let $n$ to be even. Divide the interval [$2-F_{n-2}-q_n, 2+F_{n-2}+p_n$] into three disjoint subintervals as follows:
\begin{eqnarray}
2-F_{n-2}\leq s\leq 2+F_{n-2}, \nonumber \\
2+F_{n-2}< s \leq 2+F_{n-2}+p_n \nonumber \\
\textrm{and}\quad 2-F_{n-2}-q_n\leq s < 2-F_{n-2} \nonumber
\end{eqnarray}
{\bf Case (1)}: $2-F_{n-2}\leq s\leq 2+F_{n-2}$.\\
There exists an invertible, $\lbrace 0,1 \rbrace$, upper triangular matrix $C$ of order $(n-1)$ such that $S(C^{-1})=s$. Let 
\[A=\left( \begin{array}{cc}
C & 0^T\\
0 & 0
\end{array} \right) \in\mathbb{R}^{n\times n}.\]
Then, $A$ is singular, $A^\#$ exists and  \[A^\#=\left( \begin{array}{cc}
C^{-1} & 0\\
0 & 0
\end{array} \right)\]
and so,
\begin{equation}
S(A^\#)=S(C^{-1})=s.\nonumber
\end{equation}
{\bf Case (2)}: $2+F_{n-2}< s \leq 2+F_{n-2}+p_n$.\\
 There exists $ m\in\mathbb{N}$ such that $s=2+F_{n-2}+m$ with $1\leq m \leq p_n$. Let
\[A=\left( \begin{array}{cc}
C_1 & (x^m)^T\\
0 & 0
\end{array} \right), \]
where $x^m \in \mathbb{R}^{n-1}$ is to be determined. Then one may easily verify that 
\[A^\#=\left( \begin{array}{cc}
C_1^{-1} & C_1^{-2}x^m\\
0 & 0
\end{array} \right) \]
and one has
\begin{equation}
S(A^\#) = S(C_1^{-1}) +S(C_1^{-2}x^m) \nonumber
\end{equation}
Already, $S(C_1^{-1}) = 2 + F_{n-2}$. We must determine $x^m \in \lbrace 0,1 \rbrace ^{n-1}$ such that $S(C_1^{-2}x^m) = m$. This equation is the same as,
\begin{equation}
 {\sum _{i=1}^{n-1}(e^TC_1^{-2}e^i)x_i} = m ,\label{main}
\end{equation}
with \[x^m=\left( \begin{array}{c}
x_1\\
x_2\\
x_3\\
\vdots\\
x_{n-1}
\end{array} \right), \]
where each $x_i\in \lbrace 0,1 \rbrace$, $1\leqslant i\leqslant n-1$.\\
 Since we have denoted the absolute $k$th column sum of $C_1^{-2}$ as $\alpha_k$, (\ref{main}) becomes,
\begin{equation}
\alpha_1x_1+\sum_{i=1}^{\frac{n-2}{2}}\alpha_{2i}x_{2i}+\sum_{i=1}^{\frac{n-4}{2}}(-\alpha_{2i+1})x_{2i+1}+\alpha_{n-2}x_{n-1}=m.\label{mainc1}
\end{equation}
For the sign distributions of $\alpha_k$ in the above expression, one may refer to (\ref{alpha1}) - 
(\ref{alpha(n-1)}).
We consider three steps now.\\
{\bf Step 1: } $m=p_n$. Since  $p_n=S(C_1^{-2}x^n)$,  where  $x^n=e^1+\sum_{i=1}^{\frac{n-2}{2}}e^{2i}+e^{n-1}$, one has 
$$p_n=\alpha_1+ \sum_{i=1}^{\frac{n-2}{2}}\alpha_{2i}+\alpha_{n-2}.$$ Clearly, equation (\ref{mainc1}) has a solution for $m=p_n$ from the set $ \lbrace 0,1 \rbrace $.\\
{\bf Step 2: } $m=p_n-\alpha_k$ for some $k=1,2,\cdots,{n-2}$. In this case, $$p_n=\alpha_1+ \sum_{i=1}^{\frac{n-2}{2}}\alpha_{2i}+\alpha_{n-2}.$$ Upon adding $(-\alpha_k)$ to both sides we get,\\
$$\alpha_1+ \sum_{i=1}^{\frac{n-2}{2}}\alpha_{2i}+\alpha_{n-2}-\alpha_k=p_n-\alpha_k.$$
If $k$ is even, then we can rewrite the above expression as, 
$$\alpha_1+ \sum_{i=1}^{\frac{k-2}{2}}\alpha_{2i}+\sum_{i={\frac{k+2}{2}}}^{\frac{n-2}{2}}\alpha_{2i}+ \alpha_{n-2}=p_n-\alpha_k .$$
Thus, $$ \alpha_1+ \sum_{i=1}^{\frac{k-2}{2}}\alpha_{2i}+0\cdot \alpha_k+\sum_{i={\frac{k+2}{2}}}^{\frac{n-2}{2}}\alpha_{2i}+ \sum_{i=1}^{\frac{n-4}{2}}0\cdot(-\alpha_{2i+1})+\alpha_{n-2}=p_n-\alpha_k,$$
showing that there is a solution to equation (\ref{mainc1}) from the set $\lbrace 0,1 \rbrace$ .\\
Now, if $k$ is odd, then we can rewrite the expression for $p_n-\alpha_k$ as,\\
$$\alpha_1+ \sum_{i=1}^{\frac{n-2}{2}}\alpha_{2i}+\sum_{i=1}^{\frac{k-3}{2}}0\cdot(-\alpha_{2i+1})+1\cdot(-\alpha_k)+\sum_{i={\frac{k+1}{2}}}^{\frac{n-4}{2}}0\cdot (-\alpha_{2i+1})+ \alpha_{n-2}= p_n-\alpha_k.$$
So again, a solution exists.\\
From the result of Step $2$, one may deduce that, for any $k_1,k_2,\cdots k_m$, all distinct and lying between $0$ and $n-1$, there exists a solution to equation (\ref{mainc1}) for $m=p_n-(\alpha_{k_1}+\alpha_{k_2}+\cdots \alpha_{k_m})$.\\
Since $\alpha_{n-2} $ comes twice in the expression of $p_n$, in particular for $p_n - 2\alpha_{n-2}$, equation (\ref{mainc1}) has a solution from the set $\lbrace 0,1 \rbrace $.\\
{\bf Step 3: } $m=p_n-l$, for some $l$ satisfying $p_n-\alpha_{n-2}<p_n-l<p_n$. Then $ 1\leq l < \alpha_{n-2}$ and there exists $k_1 \in \mathbb{N}$ such that $\alpha_{k_1} \leq l < \alpha_{k_1+1}$, with $k_1<{n-2}$. Now, $$p_n-l=(p_n-\alpha_{k_1})+(\alpha_{k_1}-l).$$
If $\alpha_{k_1}-l=0$, then the proof is done, by Step 2. So, let $l-\alpha_{k_1}>0$.  Then we have $ k_2 \in \mathbb{N}$ such that $\alpha_{k_2} \leq l-\alpha_{k_1} < \alpha_{k_2+1}$, with $k_2 \leq k_1$. If possible, let $k_2=k_1$ so that $2\alpha_{k_1} \leq l < \alpha_{k_1+1}$. This is a contradiction to $2\alpha_{k_1} \geq \alpha_{k_1+1}$, as shown earlier. Hence, $k_2 < k_1$.\\
Then, $$p_n-l=(p_n-\alpha_{k_1}-\alpha_{k_2})+(\alpha_{k_1}+\alpha_{k_2}-l).$$
If $\alpha_{k_1}+\alpha_{k_2}-l=0$, then the proof is done. Otherwise, there exists $ k_3 \in \mathbb{N }$ such that $$\alpha_{k_3} \leq l-(\alpha_{k_1}+\alpha_{k_2}) < \alpha_{k_3+1}.$$
Since $l-(\alpha_{k_1}+\alpha_{k_2})<l-\alpha_{k_1}$, we have $k_3 \leq k_2$. If $k_3=k_2$, then $$2\alpha_{k_2} \leq l-\alpha_{k_1} < \alpha_{k_2+1},$$ again a contradiction. Therefore, $k_3 < k_2$.
Proceeding in this manner, there exists some $i \in \mathbb{N }$ such that, $$p_n-l=p_n-(\alpha_{k_1}+\alpha_{k_2}+\cdots +\alpha_{k_i}),$$ where $\alpha_{k_1}+\alpha_{k_2}+\cdots \alpha_{k_i}-l=0$ and $k_1>k_2>\cdots > k_i$. The fact that 
$\alpha_1=\alpha_2=1$ confirms the existence of such an index $i$.

Consider the real line segment corresponding to the interval $(0,p_n]$. Since $$p_n=\alpha_1+\sum_{i=1}^{\frac{n-2}{2}}\alpha_{2i}+\alpha_{n-2},$$ this segment can be partitioned into $n$ line segments given  by $(0, \alpha_1],~(\alpha_1, \alpha_1+\alpha_2],~(\alpha_1+\alpha_2, \alpha_1+\alpha_2+\alpha_4],\cdots,~(\alpha_1+\sum_{i=1}^{\frac{n-4}{2}}\alpha_{2i}, \alpha_1+\sum_{i=1}^{\frac{n-2}{2}}\alpha_{2i}]$ and $(\alpha_1+\sum_{i=1}^{\frac{n-2}{2}}\alpha_{2i},p_n]$. These intervals have lengths $\alpha_1,~\alpha_2,~\alpha_4,~\alpha_6,\cdots,~\alpha_{n-4},~\alpha_{n-2}$ and $\alpha_{n-2}$, respectively. Note that last two intervals have the same length.

Now, we rewrite the above sub-intervals using the formula for $p_n$ in the reverse order as, $(p_n-\alpha_{n-2},p_n],~(p_n-2\alpha_{n-2}, p_n-\alpha_{n-2}],~(p_n-\alpha_{n-2}-\sum_{i=1}^2\alpha_{n-2i}, p_n-2\alpha_{n-2}],\cdots,~(p_n-\alpha_{n-2}-\sum_{i=1}^{\frac{n-2}{2}}\alpha_{n-2i}, p_n-\alpha_{n-2}-\sum_{i=1}^{\frac{n-4}{2}}\alpha_{n-2i}]$ and $(0,p_n-\alpha_{n-2}-\sum_{i=1}^{\frac{n-2}{2}}\alpha_{n-2i}]$.

We must show the existence of a solution to (\ref{mainc1}) from the set $\lbrace 0,1\rbrace$ for $m$ lying in each of these intervals. However, it has been already shown for $(p_n-\alpha_{n-2},p_n]$. For convenience, set $\gamma_n:=p_n-\alpha_{n-2}$. 
In a manner similar to the above, one may prove that, if $1\leq k \leq \frac{n-2}{2}$ and $l$ satisfies the inequalities\\
$$\gamma_n-\sum_{i=1}^{k}\alpha_{n-{2i}} < \gamma_n-\sum_{i=1}^{k-1}\alpha_{n-{2i}}-l < \gamma_n-\sum_{i=1}^{k-1}\alpha_{n-{2i}},$$
then there exists a solution to the equation (\ref{mainc1}) for $$m=p_n-\alpha_{n-2}-\sum_{i=1}^{k-1}\alpha_{n-2i}-l,$$ from the set $\lbrace 0,1 \rbrace $.

Thus, for all $m$ with $1 \leq m \leq p_n$ there exists a solution to (\ref{mainc1}) from the set $\lbrace 0,1 \rbrace $.

{\bf Case (3): }
$ 2-F_{n-2}-q_n \leq s < 2-F_{n-2} $\\
Let
\[B=\left( \begin{array}{cc}
C_3 & (y^m)^T\\
0 & 0
\end{array} \right), \]
so that
\[B^\#=\left( \begin{array}{cc}
C_3^{-1} & C_3^{-2}y^m\\
0 & 0
\end{array} \right), \] where $ y^m \in \lbrace 0,1 \rbrace ^{n-1}$. Then, $$S(B^\#)=S(C_3^{-1})+S(C_3^{-2}y^m)=2-F_{n-2}+S(C_3^{-2}y^m).$$ As $$2-F_{n-2}-q_n \leq s < 2-F_{n-2},$$ one has $s=2-F_{n-2}-m,$ for some $m$ lying between $1$ and $q_n$. So, the problem reduces to ensuring the existence of $y^m \in \lbrace 0,1 \rbrace ^{n-1}$ such that, $ S(C_3^{-2}y^m)=-m $. As before, we must determine if there exist $y_i\in \lbrace 0,1 \rbrace,~1\leq i\leq n-1$ such that
\begin{eqnarray}
{\sum _{i=1}^{n-1}(e^TC_3^{-2}e^i)y_i}& =& -m,\label{main2}
\end{eqnarray}
where \[y^m=\left( \begin{array}{c}
y_1\\
y_2\\
y_3\\
\vdots\\
y_{n-1}
\end{array} \right). \]
Since we have denoted the absolute $k$th column sum of $C_3^{-2}$ as $\beta_k$, the above equation becomes
\begin{equation}
\beta_1 x_1+\sum_{i=1}^{\frac{n-4}{2}}\beta_{2i} x_{2i} + \sum_{i=1}^{\frac{n-4}{2}}(-\beta_{2i+1}) x_{2i+1} +(-\beta_{n-3} )x_{n-2}+(-\beta_{n-3})x_{n-1}=-m.
\end{equation}
One may verify that the signs above come from the formulae for the column sums of $C_3^{-2}$. We may convert the above problem to an equivalent one as,
\begin{equation}
\beta_1 x_1+\sum_{i=1}^{\frac{n-4}{2}}\beta_{2i} x_{2i} + \sum_{i=1}^{\frac{n-4}{2}}(-\beta_{2i+1}) x_{2i+1} +(-\beta_{n-3} )x_{n-2}+(-\beta_{n-3})x_{n-1}=m\label{mainc3}
\end{equation}
and we have to show that, for any $m$ with $1\leq m \leq q_n,$ there exists a solution to the equation above, from the set $\lbrace 0,-1 \rbrace $.\\
{\bf Step 1:} $m=q_n$. Recall that, $q_n=-S(C_3^{-2}y^n)$, where $$y^n=\sum_{i=1}^{\frac{n-4}{2}}e^{2i+1}+e^{n-2}+e^{n-1}.$$ Therefore, for $m=-q_n$ there is a solution from the set $\lbrace 0,1 \rbrace$. So, for $m=q_n$ there is a solution to (\ref{mainc3}) from the set $\lbrace 0,-1 \rbrace$. \\
{\bf Step 2:} $m=q_n-\beta_k$ for some $k=1,2,\cdots {n-3}.$ We already have, $$\sum_{k=1}^{\frac{n-4}{2}}\beta_{2k+1}+\beta_{n-3}+\beta_{n-3}=q_n,$$ so that $q_n-\beta_1$ equals \\
$\beta_1(-1)+\sum_{i=1}^{\frac{n-4}{2}}\beta_{2i}\cdot 0+\sum_{i=1}^{\frac{n-4}{2}}(-\beta_{2i+1})(-1)+(-\beta_{n-3})(-1)+(-\beta_{n-3})(-1).$\\
This shows that, for $m=q_n-\beta_1,$ the equation (\ref{mainc3}) has a solution from the set $\lbrace 0,-1 \rbrace $.\\
Considering the formula for $q_n$ again, and adding $(-\beta_k)$ to both sides we get,
$$\sum_{i=1}^{\frac{n-4}{2}}\beta_{2i+1}+\beta_{n-3}+\beta_{n-3}+(-\beta_k)=q_n-\beta_k.$$
If $k$ is even, then $q_n-\beta_k$ equals \\
$0\cdot \beta_1+\sum_{i=1}^{\frac{k-2}{2}}\beta_{2i}\cdot0+
\beta_k(-1)+ \sum_{\frac{k+2}{2}}^{\frac{n-4}{2}}\beta_{2i}\cdot 0 \\
~~~~~~~~~~~~~~~~~~~~~~~~~~~~~~~~~+ \sum_{i=1}^{\frac{n-4}{2}}(-\beta_{2i+1})(-1)+(-\beta_{n-3})(-1)+(-\beta_{n-3})(-1).$\\
So, the requirement is satisfied.\\
If $k$ is odd then $q_n-\beta_k$ equals \\
$0\cdot \beta_1+\sum_{i=1}^{\frac{n-4}{2}}\beta_{2i}\cdot 0+\sum_{i=1}^{\frac{k-3}{2}}(-\beta_{2i+1})(-1)+0\cdot \beta_k \\
~~~~~~~~~~~~~~~~~~~~~~~~~+\sum_{\frac{k+1}{2}}^{\frac{n-4}{2}}(-\beta_{2i+1})(-1)+(-\beta_{n-3})(-1)+(-\beta_{n-3})(-1).$\\
Therefore, for each $k$ varying between $1$ and $n-3 $ there exists a solution for $m=q_n-\beta_k$ from the set $\lbrace 0,-1 \rbrace$. Further, it follows in an entirely similar manner, that for any $k_1,k_2,\cdots k_i$ all distinct and varying between $1$ and ${n-3}$, there exists a solution to the equation (\ref{mainc3}) for $$m=q_n-(\beta_{k_1}+\beta_{k_2}+\cdots +\beta_{k_i})$$ from the set $\lbrace 0,-1 \rbrace$. Since $\beta_{n-3} $ comes thrice in the expression of $q_n$, in particular for $q_n - 2\beta_{n-3}$ and $q_n-3\beta_{n-3}$, equation (\ref{mainc3}) has a solution from the set $\lbrace 0,1 \rbrace $.\\ 
{\bf Step 3:} $m=q_n-l$ for some $l$ satisfying $q_n-\beta_{n-3} < q_n-l < q_n $. Then $ 1\leq l< \beta_{n-3}$ and there exists $  k_1 \in \mathbb{N}$ such that $\beta_{k_1} \leq l < \beta_{k_1+1}$, with $k_1<{n-3}$.\\
Now, $$q_n-l=(q_n-\beta_{k_1})+(\beta_{k_1}-l).$$ If $\beta_{k_1}-l=0$, then the proof is done, by Step 2. So, let $l-\beta_{k_1}>0$. Then there exists $ k_2 \in \mathbb{N}$ such that $\beta_{k_2}\leq l-\beta_{k_1}< \beta_{k_2+1}$. Since $l-\beta_{k_1} < l$ we have $k_2\leq k_1$. If possible let $k_2=k_1$. Then $2\beta_{k_1}\leq l < \beta_{k_1+1}$. But already we have seen that $2\beta_{k_1} \geq \beta_{k_1+1}$. Hence, $k_2< k_1.$\\
Now, $$q_n-l=(q_n-\beta_{k_1}-\beta_{k_2})+(\beta_{k_1}+\beta_{k_2}-l).$$ If $\beta_{k_1}+\beta_{k_2}-l=0$ then the proof is done. Otherwise, $l >\beta_{k_1}+\beta_{k_2}$. There exists $k_3\in \mathbb{N}$ such that $\beta_{k_3}\leq l-(\beta{k_1}+\beta_{k_2})< \beta_{k_3+1}$. Since 
$$l-(\beta_{k_1}+\beta_{k_2})< l-\beta_{k_1}, \ \ we ~have \ \ k_3\leq k_2.$$ If possible, let $k_3=k_2$. Then $2\beta_{k_2} \leq l-\beta_{k_1} < \beta_{k_2+1}$. But, $\beta_{k_2+1}\leq 2\beta_{k_2}$, again a contradiction. So, $k_3<k_2$.\\
Now, $$q_n-l=(q_n-\beta_{k_1} -\beta_{k_2} -\beta_{k_3} )+(\beta_{k_1} +\beta_{k_2} +\beta_{k_3} -l).$$ If $\beta_{k_1} +\beta_{k_2} +\beta_{k_3} -l=0$ then there is nothing to prove. Otherwise, there exists $i\in \mathbb{N}$ such that, $$q_n-l=q_n-(\beta_{k_1}+\beta_{k_2}+\cdots +\beta_{k_i})$$ and $\beta_{k_1}+\beta_{k_2}+\cdots +\beta_{k_i}=0$ where, $k_1> k_2>\cdots>k_i$. The fact that $ \beta_1=\beta_2=1$ confirms the existence of $i$.\\
Similarly, we can prove that, for any natural number $l$ with $$q_n-2\beta_{n-3}< q_n-\beta_{n-3}-l< q_n-\beta_{n-3},$$ there exists a solution to the equation (\ref{mainc3}) for $m=q_n-\beta_{n-3}-l$ from the set $\lbrace 0,-1 \rbrace$. \\
In general, one can prove that if $1\leq k \leq \frac{n-4}{2}$ and a natural number $l$ satisfies the inequalities \\
$$\delta_n-\sum_{i=1}^k \beta_{n-(2i+1)} < \delta_n-\sum_{i=1}^{k-1}\beta_{n-(2i+1)}-l< \delta_n-\sum_{i=1}^{k-1}{\beta_{n-(2i+1)}},$$ where $\delta_n=q_n-2\beta_{n-3},$ there exists a solution to the equation $$m=q_n-2{\beta_{n-3}}-\sum_{i=1}^{k-1}{\beta_{n-(2i+1)}}-l$$ from the set $\lbrace 0,-1 \rbrace$. \\
Therefore for all integers $m$ lying between $1$ and $q_n$ there exists a solution to the equation (\ref{mainc3}) from the set $\lbrace 0,-1 \rbrace$. We have completed the proof for the case when $n$ is even. 

Let us give an argument, albeit briefly, to show that an entirely similar process applies for odd $n$. When $n$ is odd, the interval $[2-F_{n-2}-s_n, 2+F_{n-2}+r_n]$ will be divided into three sub-intervals, as earlier. Then Case $1$ proceeds in an entirely similar manner. In Case $2$, the matrix $C_1$ should be replaced by $C_2$ in the block matrix $A$. Then equation (80) will be written in the form of (81) using the column sums of $C_2^{-2}$. Then, instead of $p_n$, one should use $r_n$ and the process proceeds along similar lines. In Case $3$, the block entry in $C_3$ in the matrix $B$ should be replaced by $C_4$ and the analysis follows in an exact same manner. Hence the entire process for the odd case proceeds along similar lines to the case when $n$ is even.

This completes the proof. 
\end{proof}

\begin{remarks}
The converse of Theorem (\ref{mainresult}) is not true. Le $x^{n-1}$ and $z^{n-1}$ be as defined there. First let $n$ be even $(n \geq 8)$. Consider the $n \times n$ matrix
\[P=\left( \begin{array}{ccc}
C_4  & z^{n-1} & z^{n-1}\\
 0 & 0 & 0\\
  0 & 0 & 0
\end{array} \right), \]
where $C_4 \in \mathbb{R}^{(n-2)\times (n-2)}$. Then $P$ is an upper triangular, $\lbrace 0,1 \rbrace$, group invertible, singular matrix. One may verify that $S(P^\#)= 2-F_{n-3}-2s_{n-1}$. We have,\\
$S(P^\#)-( 2-F_{n-2}-q_n)$
\begin{eqnarray}
&=&F_{n-4}+q_n-2s_{n-1}\nonumber\\
&=& F_{n-4}+\sum_{k=1}^{\frac{n-4}{2}}\beta_{2k+1}+\beta_{n-3}+\beta_{n-3}-2\sum_{k=1}^{\frac{n-4}{2}}\beta_{2k+1}-2\beta_{n-3}\nonumber\\
&=&F_{n-4}-\sum_{k=1}^{\frac{n-4}{2}}\beta_{2k+1}\nonumber\\
&=&F_{n-4}-\sum_{k=1}^{\frac{n-6}{2}}\beta_{2k+1}-\beta_{n-3}\nonumber
\end{eqnarray}
But, $\beta_{n-3}=\vert e^T(C_3^{-2})e^{n-3}\vert >F_{n-4}$, showing that $S(P^\#)< 2-F_{n-2}-q_n$.

Next, consider the case when $n$ is odd $(n \geq 7)$. Consider the $n\times n$ matrix
\[Q=\left( \begin{array}{ccc}
C_1 & x^{n-1} & x^{n-1}\\
 0 & 0 & 0\\
  0 & 0 & 0
\end{array} \right), \]
where $C_1 \in \mathbb{R}^{(n-2)\times (n-2)}$. Then $Q$ is an upper triangular, $\lbrace 0,1 \rbrace$, group invertible, singular matrix. Again, one may verify that $S(Q^\#)= 2+F_{n-3}+2p_{n-1}$. Now,
\begin{eqnarray}
S(Q^\#)-2+F_{n-2}+r_n&=& -F_{n-4}+2p_{n-1}-r_n\nonumber\\
&=& -F_{n-4}+2\alpha_1+2\sum_{i=1}^{\frac{n-3}{2}}\alpha_{2i}+2\alpha_{n-3}-\alpha_1\nonumber
\end{eqnarray}
\begin{eqnarray}
&-& \sum_{i=1}^{\frac{n-3}{2}}\alpha_{2i}-\alpha_{n-3}-\alpha_{n-3}\nonumber\\
&=& -F_{n-4}+\alpha_1+\sum_{i=1}^{\frac{n-3}{2}}\alpha_{2i}\nonumber\\
&=&-F_{n-4}+\alpha_1+\sum_{i=1}^{\frac{n-5}{2}}\alpha_{2i}+\alpha_{n-3}\nonumber
\end{eqnarray}
However, $\alpha_{n-3}=\vert e^TC_2^{-2}e^{n-3} \vert >F_{n-4}$, showing that $S(Q^\#)> 2+F_{n-2}+r_n$.

For $n=6$, consider 
\[A=\left( \begin{array}{cccccc}
1 & 0 & 1 & 1 & 0 & 0\\
0 & 1 & 1 & 1 & 0 & 0\\
 0 & 0 & 1 & 0 & 1 & 1\\
  0 & 0 & 0 & 1 & 1 & 1\\
  0 & 0 & 0 & 0 & 0 & 0 \\
  0 & 0 & 0 & 0 & 0 & 0 
\end{array} \right). \]
Then, \[A^\#=\left( \begin{array}{cccccc}
1 & 0 & -1 & -1 & -4 & -4\\
0 & 1 & -1 & -1 & -4 & -4\\
 0 & 0 & 1 & 0 & 1 & 1\\
  0 & 0 & 0 & 1 & 1 & 1\\
  0 & 0 & 0 & 0 & 0 & 0 \\
  0 & 0 & 0 & 0 & 0 & 0 
\end{array} \right). \] 
$S(A^\#)=-12$. However, $2-F_4-q_6=-1-q_6$ where $q_6=3\beta_3=9$, by the formula for $q_n$. Hence, $S(A^\#)<2-F_4-q_6$.
\end{remarks}

\section{Appendix}
{\bf Proof of Lemma (\ref{c1}):} 
 \begin{eqnarray}
 C_1X(e^1)=C_1e^1=e^1 \quad \textrm{and}\quad C_1X(e^2)=C_1e^2=e^2.\nonumber
  \end{eqnarray}
  Let $k$ be odd, ($3\leq k \leq n-3$). Then
  \begin{eqnarray}    
 C_1X(e^k)&=&C_1(-F_{k-2} e^1 + \sum_{i=2}^{k-1}{(-1)^{i+1}F_{k-i}e^i + e^k})\nonumber\\
 &=&-F_{k-2}(C_1e^1) + \sum_{i=2}^{k-1}{(-1)^{i+1}F_{k-i}(C_1e^i) + (C_1e^k)}\nonumber\\
&=&-F_{k-2}e^1 - \sum_{i=1}^{\frac{k-1}{2}}F_{k-2i}(C_1e^{2i}) + \sum_{i=1}^{\frac{k-3}{2}}F_{k-(2i+1)}(C_1e^{2i+1}) + (C_1e^k)\nonumber\\
&=&e^1(-F_{k-2}+\sum_{p=1}^{\frac{k-3}{2}}{F_{2p}}+1) + \sum_{r=1}^{\frac{k-3}{2}}{e^{2r+1}(F_{k-(2r+1)} - \sum_{p=1}^{\frac{k-(2r+1)}{2}}F_{2p-1})}\nonumber
\end{eqnarray}
 \begin{eqnarray}
  &+&\sum_{r=1}^{\frac{k-1}{2}}{e^{2r}(-F_{k-(2r)}+\sum_{p=1}^{\frac{k-(2r+1)}{2}}{F_{2p}}+1)} +  e^k.
\nonumber\\
&=&e^k\nonumber
\end{eqnarray}
Let $k$ be even, ($4\leq k \leq n-2 $). Then,\\
$C_1X(e^k)$
\begin{eqnarray}
&=&C_1(F_{k-2} e^1 + \sum_{i=2}^{k-1}{(-1)^{i}F_{k-i}e^i + e^k})\nonumber\\
&=&F_{k-2}(C_1e^1) + \sum_{i=1}^{\frac{k-2}{2}}F_{k-2i}(C_1e^{2i}) - \sum_{i=1}^{\frac{k-2}{2}}F_{k-(2i+1)}(C_1e^{2i+1}) + (C_1e^k)\nonumber\\
&=&e^1(F_{k-2}-\sum_{p=1}^{\frac{k-2}{2}}{F_{2p-1}}) + \sum_{r=1}^{\frac{k-2}{2}}{e^{2r+1}(-F_{k-(2r+1)} + \sum_{p=1}^{\frac{k-(2r+2)}{2}}F_{2p}+1)} \nonumber\\
&+& \sum_{r=1}^{\frac{k-2}{2}}{e^{2r}(F_{k-(2r)}-\sum_{p=1}^{\frac{k-(2r)}{2}}{F_{2p-1}})} +  e^k\nonumber\\
&=&e^k.\nonumber
\end{eqnarray}
$ C_1X(e^{n-1})$
 \begin{eqnarray}
&=&C_1(F_{n-4}e^1 + \sum_{i=2}^{n-3}{(-1)^{i}F_{n-2-i}e^i} + e^{n-1})\nonumber\\
&=&F_{n-4}(C_1e^1) + \sum_{i=2}^{n-3}{(-1)^{i}F_{n-2-i}(C_1e^i) + (C_1e^{n-1})}\nonumber\\
&=&F_{n-4}(e^1) + \sum_{i=1}^{\frac{n-4}{2}}F_{n-2-2i}(C_1e^{2i}) - \sum_{i=1}^{\frac{n-4}{2}}F_{n-2-(2i+1)}(C_1e^{2i+1}) + (C_1e^{n-1})\nonumber\\
&=&e^1(F_{n-4}-\sum_{p=1}^{\frac{n-4}{2}}{F_{2p-1}}) + \sum_{r=1}^{\frac{n-4}{2}}{e^{2r+1}(-F_{n-2-(2r+1)} + \sum_{p=1}^{\frac{n-2-(2r+2)}{2}}F_{2p}+1)} \nonumber
\end{eqnarray}
\begin{eqnarray}
&+& \sum_{r=1}^{\frac{n-4}{2}}{e^{2r}(F_{n-2-(2r)}-\sum_{p=1}^{\frac{n-2-(2r)}{2}}{F_{2p-1}})} +  e^{n-1}\nonumber\\
&=&e^{n-1}.\nonumber
\end{eqnarray}
Again,
\begin{eqnarray}
XC_1(e^1)=X(e^1)=e^1\quad \textrm{and}\quad XC_1(e^2)=X(e^2)=e^2.\nonumber
\end{eqnarray}
Let $k$ be odd, ($3\leq k \leq n-3$). Then,
\begin{eqnarray}
XC_1(e^k)&=&X(e^1+\sum_{i=1}^{\frac{k-1}{2}}e^{2i}+e^k)\nonumber \\
&=& e^1+\sum_{i=1}^{\frac{k-1}{2}} Xe^{2i}+Xe^k\nonumber \\
&=& e^1+\sum_{i=1}^{\frac{k-1}{2}}(F_{2i-2}e^1+\sum_{j=2}^{2i-1}(-1)^j F_{2i-j}e^j+e^{2i})\nonumber\\&+&(-F_{k-2}e^1+\sum_{i=2}^{k-1}(-1)^{i+1}F_{k-i}e^i+e^k)\nonumber\\
&=& e^1(1+\sum_{j=1}^{\frac{k-3}{2}}F_{2j}-F_{k-2})+\sum_{r=1}^{\frac{k-1}{2}}e^{2r}(1+\sum_{j=1}^{\frac{k-(2r+1)}{2}}F_{2j}-F_{k-2r})\nonumber\\
 &+ & \sum_{r=1}^{\frac{k-3}{2}}e^{2r+1}(F_{k-(2r+1)}-\sum_{j=1}^{\frac{k-(2r+1)}{2}}F_{2j-1})+e^k\nonumber\\
 &=& e^k.\nonumber
\end{eqnarray}
When $k$ is even, ($4\leq k \leq n-2$), we have:
\begin{eqnarray}
XC_1(e^k)&=& X(\sum_{i=1}^{\frac{k-2}{2}}e^{2i+1}+e^k)\nonumber\\
&=&\sum_{i=1}^{\frac{k-2}{2}}Xe^{2i+1}+Xe^k\nonumber\\
&=&\sum_{i=1}^{\frac{k-2}{2}}(-F_{2i-1}e^1+\sum_{j=2}^{2i} (-1)^{j+1}F_{2i-j+1}e^j+e^{2i+1})\nonumber\\
&+& (F_{k-2}e^1+\sum_{i=2}^{k-1}(-1)^iF_{k-i}e^i+e^k)\nonumber
\end{eqnarray}
\begin{eqnarray}
&=& e^1(-\sum_{i=1}^{\frac{k-2}{2}}F_{2i-1}+F_{k-2})+\sum_{r=1}^{\frac{k-2}{2}}e^{2r}(F_{k-2r}-\sum_{j=1}^{\frac{k-2r}{2}}F_{2j-1})\nonumber\\
&+& \sum_{r=1}^{\frac{k-2}{2}}e^{2r+1}(1+\sum_{j=1}^{\frac{k-(2r+2)}{2}}F_{2j}-F_{k-(2r+1)})+e^k\nonumber\\
&=& e^k.\nonumber\\
XC_1(e^{n-1})&=&X(\sum_{i=1}^{\frac{n-4}{2}}e^{2i+1}+e^{n-1})\nonumber\\
&=&\sum_{i=1}^{\frac{n-4}{2}}Xe^{2i+1}+Xe^{n-1}\nonumber\\
&=&\sum_{i=1}^{\frac{n-4}{2}}(-F_{2i-1}e^1+\sum_{j=2}^{2i} (-1)^{j+1}F_{2i-j+1}e^j+e^{2i+1})\nonumber\\
&+& (F_{n-4}e^1+\sum_{i=2}^{n-3}(-1)^iF_{n-2-i}e^i+e^{n-1})\nonumber\\
&=& e^1(-\sum_{i=1}^{\frac{n-4}{2}}F_{2i-1}+F_{n-4})+\sum_{r=1}^{\frac{n-4}{2}}e^{2r}(F_{n-2-2r}-\sum_{j=1}^{\frac{n-2-2r}{2}}F_{2j-1})\nonumber\\
&+& \sum_{r=1}^{\frac{n-4}{2}}e^{2r+1}(1+\sum_{j=1}^{\frac{n-2-(2r+2)}{2}}F_{2j}-F_{n-2-(2r+1)})+e^{n-1}\nonumber\\
&=& e^{n-1}.\nonumber
\end{eqnarray}
This completes the proof of Lemma (\ref{c1}). 

The proofs of Lemma (\ref{c2}), Lemma (\ref{c3}) and Lemma (\ref{c4}) are similar and are skipped.

\noindent {\bf Proof of Lemma (\ref{colsum1}):} \\
Using the formulae for $C_1^{-1}$ from Lemma (\ref{c1}), one has 
\begin{eqnarray}
e^TC_1^{-1}e^1&=&e^Te^1=1\quad \textrm{and}\quad e^TC_1^{-1}e^2=e^Te^2=1.\nonumber
\end{eqnarray}
For odd $k$ ($3\leq k \leq n-3$),
\begin{eqnarray}
e^TC_1^{-1}e^k&=&-F_{k-2}+\sum_{i=2}^{k-1}{(-1)^{i+1}F_{k-i}+1}\nonumber\\
&=&-F_{k-2}-\sum_{i=1}^{\frac{k-3}{2}}{(F_{k-2i}-F_{k-(2i+1)})}\nonumber\\
&=&-F_{k-2}-\sum_{i=1}^{\frac{k-3}{2}}{F_{k-(2i+2)}}\nonumber\\
&=&-\sum_{i=1}^{\frac{k-1}{2}}{F_{k-2i}}\nonumber\\
&=&-F_{k-1}.  \nonumber
\end{eqnarray}
for even $k$ ($4\leq k\leq n-2$),
\begin{eqnarray}
e^TC_1^{-1}e^k&=&F_{k-2}+\sum_{i=2}^{k-1}{(-1)^{i}F_{k-i}+1}\nonumber\\
&=&F_{k-2}+\sum_{i=1}^{\frac{k-2}{2}}{(F_{k-2i}-F_{k-(2i+1)})}+1\nonumber
\end{eqnarray}
\begin{eqnarray}
&=&F_{k-2}+\sum_{i=1}^{\frac{k-2}{2}}F_{k-(2i+2)}+1\nonumber\\
&=&\sum_{i=1}^{\frac{k}{2}}{F_{k-2i}}+1\nonumber\\&=&F_{k-1}.\nonumber
\end{eqnarray}
\begin{eqnarray}
e^TC_1^{-1}e^{n-1}&=&F_{n-4}+\sum_{i=2}^{n-3}{(-1)^{i}F_{n-2-i}+1}\nonumber\\
&=&F_{n-4}+\sum_{i=1}^{\frac{n-4}{2}}{(F_{n-2-2i}-F_{n-2-(2i+1)})+1}\nonumber\\
&=&F_{n-4}+\sum_{i=1}^{\frac{n-4}{2}}F_{n-2-(2i+2)}+1\nonumber
\end{eqnarray}
\begin{eqnarray}
&=&\sum_{i=1}^{\frac{n-2}{2}}{F_{n-2-2i}}+1\nonumber\\
&=&F_{n-3}.\nonumber
\end{eqnarray}
By similar computations, the formulae involving $C_2^{-1}$, $C_3^{-1}$ and $C_4^{-1}$ may be shown to hold.

{\bf Proof of Lemma (\ref{sumofall}):} \\
Using Lemma (2.5),
\begin{eqnarray}
S(C_1^{-1})&=&\sum_{i=1}^{n-1} e^TC_1^{-1}e^i\nonumber\\
&=&2-\sum_{i=1}^{\frac{n-4}{2}}F_{2i}+\sum_{i=1}^{\frac{n-4}{2}}F_{2i+1}+F_{n-3}\nonumber\\
&=&2+\sum_{i=1}^{\frac{n-4}{2}}(F_{2i+1}-F_{2i})+F_{n-3}\nonumber\\
&=&2+\sum_{i=1}^{\frac{n-4}{2}}F_{2i-1}+F_{n-3}\nonumber\\
&=&2+F_{n-4}+F_{n-3}\nonumber\\
&=&2+F_{n-2}.\nonumber
\end{eqnarray}
Again by Lemma (2.5),
\begin{eqnarray}
S(C_2^{-1})&=&\sum_{i=1}^{n-1} e^TC_2^{-1}e^i\nonumber\\
&=&2-\sum_{i=1}^{\frac{n-5}{2}}F_{2i}+\sum_{i=1}^{\frac{n-5}{2}}F_{2i+1}+2F_{n-4}\nonumber\\
&=&2+\sum_{i=1}^{\frac{n-5}{2}}F_{2i-1}+2F_{n-4}\nonumber\\
&=&2+F_{n-5}+F_{n-4}+F_{n-4}\nonumber\\
&=&2+F_{n-2}.\nonumber
\end{eqnarray}
In a similar manner, using Lemma (2.6), one may prove that $S(C_3^{-1})=S(C_4^{-1})=2-F_{n-2}$.

{\bf Proof of Lemma (\ref{twoseqres}):}\\
We present proofs for the first set of inequalities. The proofs for the second part are similar and will be omitted. \\
It is clear that $\alpha_1=\alpha_2$. For $k$ even,
\begin{eqnarray}
\alpha_{k+1}-\alpha_k&=&\vert e^TC_1^{-2}e^{k+1}\vert - \vert e^TC_1^{-2}e^k \vert \nonumber\\
&=&F_{k-1}+\sum_{i=1}^{\frac{k}{2}}F_{2i-1}F_{k+1-2i}+\sum_{i=1}^{\frac{k-2}{2}}F_{2i}F_{k+1-(2i+1)}+F_k\nonumber\\
&-&\lbrace F_{k-2}+\sum_{i=1}^{\frac{k-2}{2}}F_{2i-1}F_{k-2i}+\sum_{i=1}^{\frac{k-2}{2}}F_{2i}F_{k-(2i+1)}+F_{k-1}\rbrace\nonumber\\
&=&F_k-F_{k-2}+2\sum_{i=1}^{\frac{k-2}{2}}F_iF_{k-i}+F_{\frac{k}{2}}^2-2\sum_{i=1}^{\frac{k-2}{2}}F_iF_{k-(i+1)}\nonumber\\
&=&F_{k-1}+2\sum_{i=1}^{\frac{k-2}{2}}F_i(F_{k-i}-F_{k-(i+1)})+F_{\frac{k}{2}}^2\nonumber\\
&=&F_{k-1}+2\sum_{i=1}^{\frac{k-2}{2}}F_iF_{k-(i+2)}+F_{\frac{k}{2}}^2.
\end{eqnarray}
Clearly,\quad $\alpha_{k+1}-\alpha_k>0$.\\
For $k$ odd,
\begin{eqnarray}
\alpha_{k+1}-\alpha_k&=&\vert e^TC_1^{-2}e^{k+1}\vert - \vert e^TC_1^{-2}e^k \vert \nonumber\\
&=&F_{k-1}+\sum_{i=1}^{\frac{k-1}{2}}F_{2i-1}F_{k+1-2i}+\sum_{i=1}^{\frac{k-1}{2}}F_{2i}F_{k+1-(2i+1)}+F_k\nonumber\\
&-&\lbrace F_{k-2}+\sum_{i=1}^{\frac{k-1}{2}}F_{2i-1}F_{k-2i}+\sum_{i=1}^{\frac{k-3}{2}}F_{2i}F_{k-(2i+1)}+F_{k-1}\rbrace\nonumber\\
&=&(F_k-F_{k-2})+2\sum_{i=1}^{\frac{k-1}{2}}F_iF_{k-i}-2\sum_{i=1}^{\frac{k-3}{2}}F_iF_{k-(i+1)}-F_{\frac{k-1}{2}}^2\nonumber\\
&=&F_{k-1}+2\sum_{i=1}^{\frac{k-3}{2}}F_i(F_{k-i}-F_{k-(i+1)})+2F_{\frac{k-1}{2}}F_{\frac{k+1}{2}}-F_{\frac{k-1}{2}}^2\nonumber
\end{eqnarray}
\begin{eqnarray}
&=&F_{k-1}+2\sum_{i=1}^{\frac{k-3}{2}}F_iF_{k-(i+2)}+F_{\frac{k-1}{2}}(F_{\frac{k+1}{2}}-F_{\frac{k-1}{2}})\nonumber\\
&+&F_{\frac{k-1}{2}}F_{\frac{k+1}{2}}.
\end{eqnarray}
Clearly, $\alpha_{k+1}-\alpha_k>0$.\\
Now, for even $n$, $2\alpha_1>\alpha_1=\alpha_2$.\\
Next, we prove : $2\alpha_k\geq \alpha_{k+1}$.\\
For $k$ even,
$(\alpha_{k+1}-\alpha_k)-\alpha_k$
\begin{eqnarray}
&=&F_{k-1}+2\sum_{i=1}^{\frac{k-2}{2}}F_iF_{k-(i+2)}+F_{\frac{k}{2}}^2-\lbrace F_{k-2}+\sum_{i=1}^{\frac{k-2}{2}}F_{2i-1}F_{k-2i}\nonumber\\
&+&\sum_{i=1}^{\frac{k-2}{2}}F_{2i}F_{k-(2i+1)}+F_{k-1}\rbrace\nonumber\\
&=&2\sum_{i=1}^{\frac{k-2}{2}}F_iF_{k-(i+2)}+F_{\frac{k}{2}}^2-F_{k-2}-2\sum_{i=1}^{\frac{k-2}{2}}F_iF_{k-(i+1)}\nonumber\\
&=&-2\sum_{i=1}^{\frac{k-2}{2}}F_i(F_{k-(i+1)}-F_{k-(i+2)})-(F_{k-2}-F_{\frac{k}{2}})^2\nonumber
\end{eqnarray}
[If $k=4$, we obtain a negative value and the proof is done. If $k>4$, we go to the next step]
\begin{eqnarray}
&=&-2\sum_{i=1}^{\frac{k-2}{2}}F_iF_{k-(i+3)}-F_{k-2}+(F_{\frac{k-2}{2}}+F_{\frac{k-4}{2}})(F_{\frac{k-2}{2}}+F_{\frac{k-4}{2}})\nonumber\\
&=&-F_{k-2}-2\sum_{i=1}^{\frac{k-4}{2}}F_iF_{k-(i+3)}+F_{\frac{k-2}{2}}^2+F_{\frac{k-4}{2}}^2\nonumber
\end{eqnarray}
[If $k=6$, we obtain a negative value and the proof is done. If $k>6$, we go to the next step.]
\begin{eqnarray}
&=&-F_{k-2}-2\sum_{i=1}^{\frac{k-6}{2}}F_iF_{k-(i+3)}-2F_{\frac{k-4}{2}}F_{\frac{k-2}{2}}+F_{\frac{k-2}{2}}^2+F_{\frac{k-4}{2}}^2\nonumber\\
&=&-F_{k-2}-2\sum_{i=1}^{\frac{k-6}{2}}F_iF_{k-(i+3)}+(F_{\frac{k-2}{2}}-F_{\frac{k-4}{2}})^2\nonumber
\end{eqnarray}
\begin{eqnarray}
&=&-F_{k-2}-2\sum_{i=1}^{\frac{k-8}{2}}F_iF_{k-(i+3)}-2F_{\frac{k-6}{2}}F_{\frac{k}{2}}+F_{\frac{k-6}{2}}^2\nonumber\\
&=&-F_{k-2}-2\sum_{i=1}^{\frac{k-8}{2}}F_iF_{k-(i+3)}-F_{\frac{k-6}{2}}F_{\frac{k}{2}}-F_{\frac{k-6}{2}}(F_{\frac{k}{2}}-F_{\frac{k-6}{2}}).\nonumber
\end{eqnarray}
So, for $k=8$ and consecutive even values, $\alpha_{k+1}-2\alpha_k\leq 0$.\\
Observe that for $k=2$, the formula above does not work. However, $\alpha_1=\alpha_2=1$ and $\alpha_3=2$. So, $2\alpha_2=\alpha_3$.\\
Now for $k$ odd, $(\alpha_{k+1}-\alpha_k)-\alpha_k$
\begin{eqnarray}
&=&F_{k-1}+2\sum_{i=1}^{\frac{k-3}{2}}F_iF_{k-(i+2)}+F_{\frac{k-1}{2}}F_{\frac{k-3}{2}}+F_{\frac{k-1}{2}}F_{\frac{k+1}{2}}\nonumber\\
&-&\lbrace F_{k-2}+\sum_{i=1}^{\frac{k-1}{2}}F_{2i-1}F_{k-2i}+\sum_{i=1}^{\frac{k-3}{2}}F_{2i}F_{k-(2i+1)}+F_{k-1}\rbrace\nonumber\\
&=&-F_{k-2}+2\sum_{i=1}^{\frac{k-3}{2}}F_iF_{k-(i+2)}+F_{\frac{k-1}{2}}F_{\frac{k-3}{2}}+F_{\frac{k-1}{2}}F_{\frac{k+1}{2}}\nonumber\\
&-&2\sum_{i=1}^{\frac{k-3}{2}}F_iF_{k-(i+1)}-F_{\frac{k-1}{2}}^2\nonumber\\
&=&-F_{k-2}-2\sum_{i=1}^{\frac{k-3}{2}}F_iF_{k-(i+3)}+F_{\frac{k-1}{2}}F_{\frac{k-3}{2}}+F_{\frac{k-1}{2}}(F_{\frac{k-1}{2}}+F_{\frac{k-3}{2}})\nonumber\\&-&F_{\frac{k-1}{2}}^2\nonumber\\
&=&-F_{k-2}-2\sum_{i=1}^{\frac{k-3}{2}}F_iF_{k-(i+3)}+2F_{\frac{k-1}{2}}F_{\frac{k-3}{2}}\nonumber
\end{eqnarray}
[If $k=3$ or $k=5$, one obtains a negative value and the proof is done. Otherwise, we go to the next step.]
\begin{eqnarray}
&=&-F_{k-2}-2\sum_{i=1}^{\frac{k-5}{2}}F_iF_{k-(i+3)}-2F_{\frac{k-3}{2}}^2+2F_{\frac{k-1}{2}}F_{\frac{k-3}{2}}\nonumber\\
&=&-F_{k-2}-2\sum_{i=1}^{\frac{k-5}{2}}F_iF_{k-(i+3)}+2F_{\frac{k-3}{2}}(F_{\frac{k-1}{2}}-F_{\frac{k-3}{2}})\nonumber
\end{eqnarray}
\begin{eqnarray}
&=&-F_{k-2}-2\sum_{i=1}^{\frac{k-7}{2}}F_iF_{k-(i+3)}-2F_{\frac{k-5}{2}}F_{\frac{k-1}{2}}+2F_{\frac{k-3}{2}}F_{\frac{k-5}{2}}\nonumber\\
&=&-F_{k-2}-2\sum_{i=1}^{\frac{k-7}{2}}F_iF_{k-(i+3)}-2F_{\frac{k-5}{2}}(F_{\frac{k-1}{2}}-F_{\frac{k-3}{2}}).\nonumber
\end{eqnarray}
So, for $k=7$ and consecutive odd values, $\alpha_{k+1}-2\alpha_k\leq 0$.

\noindent{\bf Acknowledgements}\\
The first author acknowledges funds received from MATRICS (MTR/2018/001132) of SERB, Government of India. The authors thank Professor Ajit Iqbal Singh for her many comments and suggestions on an earlier dratft. They also thank Sushmitha and Samir Mondal for patiently verifying many of the proofs.

\end{document}